\documentclass[12pt]{amsart}
\usepackage[T1]{fontenc}
\usepackage[utf8]{inputenc}
\usepackage{color,amssymb,amsmath,mathrsfs,enumerate}
\usepackage{hyperref}
\usepackage{tikz}
\usepackage{tikz-cd}

\theoremstyle{plain}
\newtheorem{Theorem}{Theorem}[section]
\newtheorem{Lemma}[Theorem]{Lemma}
\newtheorem{conjecture}[Theorem]{Conjecture}

\newtheorem{Corollary}[Theorem]{Corollary}
\newtheorem{Example}[Theorem]{Example}
\theoremstyle{definition}
\newtheorem{Remark}[Theorem]{Remark}

\newtheorem{Definition}[Theorem]{Definition}

\begin{document}
\title{Looking for a continuous version of  Bennett--Carl theorem}
%\title{Absolutely summing inclusions of r.i. function spaces }

%\author{Sergei V. Astashkin, Karol Leśnik and Michał Wojciechowski}
\author{Sergey V. Astashkin}
\address[Astashkin]{Department of Mathematics, Samara National Research University, Mos\-kov\-skoye shosse 34, 443086, Samara, Russia; Lomonosov Moscow State University, Moscow, Russia; Moscow Center of Fundamental and Applied Mathematics, Moscow, Russia; Department of Mathematics, Bahcesehir University, 34353, Istanbul, Turkey.}
\email{astash56@mail.ru}
\urladdr{www.mathnet.ru/rus/person/8713}

\author[Le\'snik]{Karol Leśnik}
\address[Le\'snik]{
Faculty of Mathematics and Computer Science, Adam Mickiewicz University in Pozna\'n,
ul. Uniwersytetu Pozna\'nskiego 4, 61-614 Pozna\'n, Poland}
\email{klesnik@vp.pl, karles5@amu.edu.pl}

\author[Wojciechowski]{Michał Wojciechowski}
\address[Wojciechowski]{%
Institute of Mathematics, Polish Academy of Sciences, ul. \'Sniadeckich 8, 00-656 Warsaw, Poland}
\email{miwoj@impan.pl}

\maketitle

\rightline{\it Dedicated to the memory of}
\rightline{\it Professor Albrecht Pietsch}
\rightline{\it (1934--2024)}
\vspace{5 mm}

\renewcommand{\thefootnote}{\fnsymbol{footnote}}


\begin{abstract}
We study absolute summability of inclusions of r.i. function spaces. It appears that such properties are closely related, or even determined by absolute summability of inclusions of subspaces spanned by the Rademacher system in respective r.i. spaces. Our main result states that for $1<p<2$ the inclusion $X_p\subset L^p$ is $(q,1)$-absolutely summing for each $p<q<2$, where $X_p$ is the unique r.i. Banach function space in which the Rademacher system spans copy of $l^p$. This result may be regarded as a continuous version of  the well-known Carl--Bennett theorem. Two different approaches to the problem and extensive discussion on them  are presented. We also conclude summability type of a kind of Sobolev embedding in the critical case. 
\end{abstract}

\newcommand{\Exp}{{\rm Exp}\,}

\section{Introduction}

Origins of the theory of $p$--summing operators goes back to the work of Grothendieck from the 1950s, however, the precise definition and background results were proposed by Pietsch in 1967, while the ideas and their importance have been propagated by Lindenstrauss and Pełczyński \cite{LP68}.  

While the theory was widely developed and found a number of nontrivial applications in modern theory of Banach spaces, the number of concrete examples of absolutely summing operators is quite limited. 
As it is usually difficult to characterize summability properties of a given operator, a great part of the theory focuses on investigations on absolute summability of the simplest operator, i.e. the identity operator but between different Banach spaces. 
One of the most important result in this direction is the Bennett--Carl theorem which asserts that inclusion
$l^p\subset l^q$ is $(r,1)$-absolutely summing when $1\leq p\leq q\leq 2$ and $1/r=1/p-1/q+1/2$ (see \cite{Ca74,Be73}). It generalizes inequalities of Littlewood \cite{Li30}, Orlicz \cite{Or33} and Grothendieck \cite{Gro53}. 
On the other hand, it gave rise to further investigations of summability properties of inclusions of more general sequence spaces. In \cite{MM00} Maligranda and Mastyło proved an analogue of the Bennett--Carl theorem for Orlicz sequence spaces, while Defant, Mastyło and Michels generalized it to r.i. sequence spaces in a series of papers \cite{DMM01,DMM02,DMM02b}.  It is also worth to mention that  Bennett--Carl theorem was the main tool in  \cite{Woj97}, where the third author partially answered the question of Pełczyński about absolute summability of Sobolev embedding.

At the same time our knowledge about analogous properties of inclusions of function spaces is almost {\it carte blanche}.  
One reason for that may be that, in contrast to sequence spaces case, inclusions of Lebesgue  spaces  have much less attractive absolutely summing properties. It is caused mainly by their  richer  isomorphic structure. In fact, since the Rademacher system spans complemented copies of $l^2$ in $L^p[0,1]$ for $1<p<\infty$, it follows that  the inclusions
\[
L^p[0,1]\subset L^q[0,1]  \text{ for } 1<q<p<\infty 
\]
are at most of the same absolutely summing type as the inclusion $l^2\subset l^2$, i.e. $(2,1)$-absolutely summing. 
On the other hand, the inclusion $L^{\infty}[0,1]\subset L^p[0,1]$ is $p$-absolutely summing for each $1\leq p<\infty$ \cite{DJT95} (actually, by Pietsch factorization theorem, this inclusion is a kind of archetype for all  $p$-absolutely summing mappings). However,  the Rademacher system spans $l^1$ in $L^{\infty}[0,1]$ and, consequently,  the embedding $L^{\infty}[0,1]\subset L^1[0,1]$ restricted to Rademacher subspaces is equivalent to the embedding $l^1\subset l^2$, thus $1$-absolutely summing. 
This suggests that absolute summability of inclusion of two r.i. spaces is closely related, or even determined, by summability of the respective inclusion of their subspaces spanned by the Rademacher system.
Moreover, above examples indicate that absolute summability of the inclusion $X\subset L^p$ should strengthen when the function space $X$ is in a certain sense close enough to $L^{\infty}$.
More precisely, we will be interested in r.i. spaces $X$ between $L^{\infty}$ and $G$, where by  $G$ is the Rodin--Semenov space, i.e. the smallest r.i. space in which Rademacher system $(r_n)$ spans $l^2$. Thus, if $X$ is strictly smaller than $G$, the Rademacher system $(r_n)$ spans  in $X$ another sequence space than $l^2$ and one may expect better summablility of the inclusion $X\subset L^p$. Furthermore,  according to \cite{As10,As20}, for each r.i. sequence space $E\not = l^2$ which is  an interpolation space for the couple  $(l^1,l^2)$, there is exactly one  interpolation r.i. function space $X_E$ between  $L^{\infty}$ and $G$ such that $(r_n)$ spans the copy of $E$ in $X_E$.  Consequently, the natural question  arises, if the inclusion 
\[
X_E\subset L^p
\]
has the same  absolute summability type as has  that of the corresponding sequence spaces, i.e.
\[
E\subset l^2?
\]
In another words, summarizing the above discussion, we consider the question whether $s=t$ in the diagram below
\[
\begin{tikzcd}
L^{\infty}\arrow[hook]{d} \arrow[hook]{r}{(1,1)} & L^1  & l^1\arrow[hook]{d} \arrow[hook]{r}{(1,1)} & l^2\arrow[equal]{d}\\
X_E \arrow[hook]{d} \arrow[hook]{r}{(s,1)} & L^p \arrow[hook]{u} & E \arrow[hook]{d} \arrow[hook]{r}{(t,1)} & l^2\arrow[equal]{d} \\
G\arrow[hook]{r}{(2,1)} &L^2 \arrow[hook]{u}& l^2  \arrow[hook]{r}{(2,1)} & l^2 \\
\end{tikzcd}
\]
\centerline {Diagram 1}
\vspace{3mm}

In the paper we focus mainly on the case of $E=l^p$, $1<p<2$. Denoting $X_p:=X_{l^p}$, we investigate summability of the inclusion $X_p\subset L^p$ 
(optimality of the choice of the bigger space will be clear in the sequel). Then on the right diagram the inclusion $l^p\subset l^2$ is $(p,1)$-absolutely summing thanks to the Bennett--Carl theorem. Thus the question is whether  $X_p\subset L^p$  is also $(p,1)$-summing? Although we are not able to resolve the problem completely, our main result gives a partial answer to it.

\begin{Theorem}\label{main}
Let $1<p<2$. Then the inclusion $X_p\subset L^p$ is $(q,1)$ - absolutely summing for each $p<q<2$.  
\end{Theorem}

It is also worth to point out that in the proof we do not use Bennett--Carl theorem. Instead, the proof relies on continuity of the transposition operator $\mathcal{T}:f(s,t)\to f(t,s)$ between respective mixed norm spaces. It appears that $\mathcal{T}$ is bounded from $L^{\infty}(X)$ to $L^1(X)$, when $X$ is an Orlicz space, but surprisingly not for $X=X_p$ (see Example \ref{trbr}). 

Exploring the same ideas, we prove also ``Orlicz spaces'' version of the above.
\begin{Theorem}\label{mainOrlicz}
Let $1<p<2$.  For each $p<q<2$ the inclusion $ExpL^{q'}\subset L^p $ is $(l^{q,\infty},1)$-absolutely summing, where $\frac{1}{q}+\frac{1}{q'}=1$. 
\end{Theorem}

Proofs of both results together with a wide discussion of the method and its perspectives constitute the content of Section \ref{transpositionsection}. 

In the next Section \ref{AS} an analogous question is considered, but with a simplification that the destination space is the biggest possible among r.i. spaces, i.e. $L^1$. Then we present a completely different approach to the subject in which the central role is played by the following theorem. 
\begin{Theorem}\label{Th1}
There exists a universal constant $\gamma>0$ such that for every $\tau>0$, any $n\in \mathbb{N}$
and arbitrary sequence $(g_k)_{k=1}^n\subset L_1$  we can find $\epsilon_k=\pm 1$, $k=1,2,\dots,n$, with
\begin{equation}\label{equ4}
 \gamma \int_0^\tau\Big(\sum_{k=1}^n \|g_k\|_1r_k\Big)^*(t)\,dt\leq \int_0^\tau\Big(\sum_{k=1}^n \epsilon_k g_k\Big)^*(t)\,dt.
\end{equation}
\end{Theorem}

This result seems to be interesting in its own. Since the choice of signs depends on $\tau$, one may regard inequality \eqref{equ4} as a weak version of submajorization. Nevertheless, we are able to use it to deduce a number of results around summability properties of the inclusion $X\subset L^1$. 

Finally, in the last section we apply the main result to conclude absolute summability properties of embeddings of Sobolev spaces into Lebesgue spaces. Precisely, we determine  summability type of  the inclusions
\[
W^{\frac{n}{m},m}(I^n)\subset L^{\frac{n}{m}}(I^n) {\rm\ and\ } W^{\frac{n}{m},m}(I^n)\subset L^{1}(I^n).
\]
Up to our knowledge, the result is new and contributes to the Pełczyński question about summability of Sobolev embeddings (cf. \cite{Woj97}).

%%%%%%%%%%%%%%%%%%%%%%%%%%%%%%%%%%%%%%%%%%%%%%%%%%%%%%%%%%%%%%%%%%%%%%%%%%%%%%%%%%%%

\section{Notation and definitions}

\subsection{Banach function spaces}

By $L^0 = L^0(I)$ we denote the space of all equivalence 
classes of real-valued Lebesgue measurable functions defined on $I = [0, 1]$. A {\it Banach function space} 
$X = (X, \|\cdot\|_X)$ on $I$ is  a Banach space contained in $L^0(I)$, which satisfies the so-called ideal 
property, i.e.: if $f, g \in L^0, |f| \leq |g|$ almost everywhere on $I$ and $g \in X$, then also $ f\in X$ and $\|f\|_X \leq \|g\|_X$. 
Analogously we define Banach sequence spaces, i.e. replacing $I$ by $\mathbb N$ with the counting
measure.

For a Banach function space $X$  its {\it K{\"o}the dual space} $X^{\prime}$ is the space 
of those $f \in L^0$ for which the {\it dual norm}
\begin{equation} \label{dual}
\|f\|_{X\prime} = \sup_{g \in X, \, \|g\|_{X} \leq 1} \int_{I} |f(t) g(t) | \, dt
\end{equation}
is finite. The K{\"o}the dual $X^{\prime}$ of  a Banach function space is a Banach function space as well. Moreover, always
$X \subset X^{\prime \prime}$, but $X = X^{\prime \prime}$ with $\|f\|_X = \|f\|_{X^{\prime \prime}}$ 
if and only if  $X$ has the {\it Fatou property}. Equivalently, $X$ has the Fatou property  if the conditions $0 \leq f_{n} \nearrow f$ a.e. on $I$
and $\sup_{n \in {\mathbb{N}}} \|f_{n}\|_X < \infty$ imply that $f \in X$ and $\|f_{n}\|_X \nearrow \|f\|_X$.

Recall  finally that a Banach function space $X$ is called $p$-concave when there is $C>0$ such that for each $n\in \mathbb{N}$ and every sequence $(f_k)_{k=1}^n\subset X$ there holds 
\[
\Big(\sum_{k=1}^n\|f_k\|_X^p\Big)^{1/p}\leq C \Big\|\Big(\sum_{k=1}^n|f_k|^p\Big)^{1/p}\Big\|_X,
\]
(cf. \cite[Definition 1.d.3]{LT79}).

%%%%%%%%%%%%%%%%%%%%%%%%%%%%%%%%%%%%%%%%%%%%%%%%%%%%%%%%%%%%%%%%%%%%%%%%%%%%%%%%%%%%%%%%%%%%%%%%%%%%%%%%

\subsection{Rearrangement invariant spaces}
A Banach function space $X = (X,\| \cdot \|_{X})$ with the Fatou property is said to be a {\it rearrangement invariant} space (r.i. space for short)
if  for any $f\in X$ and $g \in L^{0}$ the equality  $d_{f}(\lambda)\equiv d_{g}(\lambda)$ imply $g\in X$ and $\| f\|_{X} = \| g\|_{X}$,  where
$$
d_{f}(\lambda) := m(\{t \in I: |f(t)|>\lambda \}) {\rm \ for\  }\lambda \geq 0.
$$ 
Moreover,  $\| f\|_{X}=\| f^{\ast }\|_{X}$, where 
$f^{\ast }(t) = \mathrm{\inf } \{\lambda >0\colon \ d_{f}(\lambda ) < t\},\ t\geq 0$.

For a r.i. function space $X$ its fundamental function $\varphi_X$ is defined as follows
\begin{equation*}
\varphi_X(t)=\|\chi_{[0,t]}\|_X, ~t\in I,
\end{equation*}
where  $\chi_A$ is the characteristic function of a set $A$.

A kind of Lorentz spaces will play the central role in the paper, thus
let us recall the respective definition. Given $1\leq p<\infty$ and a concave, nondecreasing function $\varphi:I\to [0,\infty)$ with $\varphi (0) = 0$,  
the {\it Lorentz space} $\Lambda_{p,\varphi}$ is given by the norm
\begin{equation*}
\|f\|_{\Lambda_{p,\varphi}} = \Big(\int_I (f^*(t))^p \,d\varphi(t)\Big)^{1/p}.
\end{equation*}
The case $p=\infty$ corresponds to the {\it Marcinkiewicz spaces} $M_{\varphi}$ given by the norm
\begin{equation} \label{Marcinkiewicz}
\|f\|_{M_{\varphi}}=\sup_{t\in I} \frac{\varphi(t)}{t} \int_0^t f^*(s) \,ds.
\end{equation}

%Denote by $M(\phi)$ the Marcinkiewicz space equipped with the norm
%$$
%\|x\|_{M(\phi)}=\sup_{0<t\le 1}\frac{1}{\phi(t)}\int_0^t x^*(s)\,ds.$$

We will be mainly interested in a special class of Lorentz spaces, which for brevity we will denote by $X_p$. Namely, for $1<p<2$ we put
\[
X_p= \Lambda_{p,W},
\]
where
\[
W(t)=\ln^{1-p}(e/t).
\]

Let us also recall the definition of another class of r.i. function spaces, i.e. Orlicz spaces.
Let $\Phi$ be an increasing convex function on $[0, \infty)$ such that $\Phi (0) = 0$. The {\it Orlicz space} $L_{\Phi}$ is defined by the Luxemburg--Nakano norm
$$
\| f \|_{L_{\Phi}} = \inf \Big\{\lambda > 0: \int_I \Phi(|f(t)|/\lambda) \, dt \leq 1 \Big\},
$$
(see e.g.  \cite{Mal89}).

For general properties of r.i. spaces we refer to the books \cite{BS88}, \cite{KPS82} and
\cite{LT79}. Once again, we will use a special notation for a special subclass of Orlicz spaces. Namely, 
\[
{\Exp}L^p:=L_{N_p} {\rm \ for\ }N_p(u)=e^{u^p}-1, u\ge 0, 1\leq p<\infty.
\]
Moreover, the Rodin--Semenov space, traditionally denoted  by $G$,
is the subspace of order continuous elements of ${\Exp}L^2$ (cf. \cite{RS75,As20,LT79}).

\subsection{Mixed norm spaces}

In the sequel we will work with mixed norm spaces. Having two Banach function spaces $X,Y$ on $I=[0,1]$, the mixed norm space $X(Y)$ is defined to be the space of all functions measurable on $I^2=I\times I$ such that
\[
\|f\|_{X(Y)}=\|\gamma\|_{X},
\]
where
\[
\gamma(t)=\|f(\cdot,t)\|_{Y}.
\]
We will use the notation $\|\|f(s,t)\|_{Y(s)}\|_{X(t)}:=\|\gamma\|_{X}$ in order to emphasize which variable corresponds to respective norm.  More information on mixed norm spaces may be found in \cite{Bu87,BBS02} and references therein.

\subsection{Calderón--Lozanovski{\v \i} spaces}

Let us recall the Calder\'on-Lozanovski{\v \i} construction for Banach function spaces. By ${\mathcal U}$ we denote the class of  functions 
$\rho: {\mathbb R_+} \times {\mathbb R_+} \rightarrow {\mathbb R_+}$,where ${\mathbb R_+}:=[0,\infty)$, that are positively homogeneous  and concave (note that any function $\rho \in {\mathcal U}$ is continuous on $(0, \infty) \times (0, \infty)$).

Given $\rho \in {\mathcal U}$ and a couple of Banach function lattices $(X_0, X_1)$ on $I$, the 
{\it Calder\'{o}n-Lozanovski{\u \i} space} $\rho (X_0, X_1)$ is defined as the set of all $f \in L^0$ such that for some $f_0 \in X_0, f_1 \in X_1$ 
with $\| f_0 \|_{X_0} \leq 1, \| f_1\|_{X_1} \leq 1$ and for some $\lambda > 0$ we have
\begin{equation*}
| f(x)| \leq \lambda\, \rho(|f_0(x)|, | f_1(x)|) ~~ {\rm a.e. ~on} ~ I. 
\end{equation*}
The norm $\| f \|_{\rho (X_0, X_1)}$ on $\rho (X_0, X_1)$ is then defined as the 
infimum of those $\lambda$ for which the above inequality holds.

The Calder\'on-Lozanovski{\v \i} construction is an interpolation method for positive operators and under additional assumptions on $X_0,X_1$ (for example the Fatou property) also an interpolation method for all linear operators (see \cite{Mal89}).
More information, especially on interpolation properties, can be found in \cite{BK91, KPS82, Mal89, Ni85}.
\vspace{3mm}

\subsection{K-method of interpolation}
Let $ (X_0, X_1)$ be a  couple of Banach function spaces. Then the {\it Peetre K-functional} of  $f \in X_0+X_1$ is
$$
K(t, f; X_0, X_1) = \inf \{ \| f_0\|_{X_0} + t \| f_1\|_{X_1}: f = f_0 + f_1, f_0 \in X_0, f_1 \in X_1 \} {\rm \ for\ }t > 0. 
$$
If  $E$ is a Banach function space on $(0, \infty)$ containing the function $\min(1, t)$ then  the space $(X_0, X_1)_E^K$ of the real $K$-method of interpolation is defined by
$$
(X_0, X_1)_E^K = \{f \in X_0 + X_1: K(t, f; X_0, X_1) \in E \}
$$
with the norm $$\| f \|_{(X_0, X_1)_E^K} = \| K(\cdot, f; X_0, X_1) \|_E.$$ 

For the special case of $E$, i.e.
$$
\| f\|_E = \left(\int_0^{\infty} (t^{-\theta} |f(t)|)^p \frac{dt}{t}\right)^{1/p}$$ for $0 < \theta < 1$ and $ 1 \leq p < \infty$ or 
$$
\| f\|_E = \sup_{t > 0} t^{-\theta} |f(t)|$$ 
for $0 \leq \theta \leq 1$ (when $p= \infty$),
we obtain the classical spaces of the real K-method $(X_0, X_1)_{\theta, p}$.
More on the K-method and  interpolation spaces  may be found in the books \cite{BS88,  BK91} and \cite{KPS82}.

\subsection{Rademacher sequence in r.i. spaces}
We will be mainly interested in r.i. spaces that are in a certain sense close to $L^{\infty}$, i.e. are smaller than $G$. 
It is due to the fact that in r.i. spaces between  $L^{\infty}$ and  $G$, the Rademacher system spans copies of r.i. sequence spaces smaller than $l^2$. 
More precisely, by Astashkin results \cite[Theorem 3.2 and 3.4]{As20} (see also \cite{As10}) for each interpolation space  $E$ between $l^1$ and $l^2$, there is (unique) interpolation r.i. function space $X_E$ between $L^{\infty}$ and  $G$, such that the Rademacher system spans a copy of $E$ in $X_E$. Since the couple $(l^1,l^2)$ is K-monotone, every interpolation space for this couple can be represented as a space of the real K-method of interpolation. Moreover, denoting by $S$ the corresponding  parameter, we have the following: 
\begin{equation}\label{Asidethentity}
{\rm \ \ if\ \ }E=(l^1,l^2)_S^K, {\rm \ \ then\ \ }X_E=(L^{\infty},G)_S^K.
\end{equation}

From the above, it follows that
\begin{itemize}
\item for $E=l^p$, $1<p<2$,
\[
X_{l^p}=X_p=(L^{\infty},G)_{\theta,p} {\rm \ since\ }(l^1,l^2)_{\theta,p}=l_p,
\]
where $\theta =\frac{2(p - 1)}{p}$ (see \cite{RS75} or \cite[Example 3.1]{As20}),
\item  for $E=l^{p,\infty}$, $1<p<2$,
\[
X_{l^{p,\infty}}={\Exp}L^{p'}=(L^{\infty},G)_{\theta,\infty}  {\rm \ since\ }(l^1,l^2)_{\theta,\infty}=l_{p,\infty},
\]
where $p'$ is conjugate of $p$ and $\theta =\frac{2(p - 1)}{p}$ (see \cite{RS75}, 
\cite{MP-84} or \cite[Example 3.2]{As20}).
\end{itemize}

%Finally, we need to mention that $X_p$ spaces appear naturally also in another contexts as Lorentz--Zygmund spaces and  are denoted by $L^{\infty,p,-1}$. 

\subsection{Absolutely summing operators}

A linear operator $T:X\to Y$ between Banach spaces is called  $(p,q)$-absolutely
summing (or simply:  $(p,q)$-summing) if  for some constant $C>0$ and each finite collection of vectors  $x_1,x_2,...,x_n\in X$,
\[
(\sum_{k=1}^{n}\|Tx_i\|^p_Y)^{\frac{1}{p}}\leq C\sup_{x^*\in B_{X^*}}\left(\sum_{i=n}^n|\langle x^*,x_i\rangle|^q\right)^{1/q}.
\]
$(p,p)$-summing operators are shortly called $p$-summing.

%One can regard $(p,q)$-summing operators as a generalization of Schatten classes to a non-Hilbertian context, since   Furhtermore, 
%Except some particular cases, it is rather difficult to establish summability properties of a given operator. On the other hand, a number of significant applications comes from considering summablility types of the identity operator. 

In \cite{DMM01, DMM02, DMM02b} the following natural generalization of $(p,q)$-summability was proposed.

\begin{Definition}
Let $E$ be a Banach sequence space such that $l^p\subset E$ for some $1\leq p<\infty$. An operator $T:X\to Y$ between Banach spaces $X,Y$ is called {\it $(E,p)$-absolutely summing} (or  $(E,p)$-summing, for short) if there is a constant $C>0$ such that for each finite collection of vectors $x_1,x_2,...,x_n\in X$,
\[
\|(\|Tx_i\|_Y)_{i=1}^{n}\|_E\leq C\sup_{x^*\in B_{X^*}}\Big(\sum_{i=n}^n|\langle x^*,x_i\rangle|^p\Big)^{1/p}.
\]
\end{Definition}

More information on absolute summability and its applications the reader can find in books \cite{DJT95}, \cite{TJ89} and  \cite{AK06}.

%%%%%%%%%%%%%%%%%%%%%%%%%%%%%%%%%%%%%%%%%%55
%\section{The main conjecture}

\section{Inclusion of $X_p$ into $L^p$}\label{transpositionsection}

The proof of Theorem \ref{main} is based on the following lemma about boundedness of transposition operator $\mathcal{T}$ on mixed norm spaces.

\begin{Lemma}\label{lemtransp2}
Let $1<p<2$ and consider the transposition operator 
$$
\mathcal{T}:f(s,t)\mapsto f(t,s),
$$
defined on measurable functions on $I^2$. Then $\mathcal{T}:L^{\infty}({\Exp}L^{\alpha})\to   L^{p}({\Exp}L^{\alpha})$ is bounded for each $\alpha>0$. 
\end{Lemma}
\proof
Clearly, in both cases 
$$
\mathcal{T}:L^{\infty}(L^{p})\to   L^{p}(L^{p})
$$
and 
$$
\mathcal{T}:L^{\infty}(L^{\infty})\to   L^{p}(L^{\infty})
$$
$\mathcal{T}$  is a contraction. On the other hand, it is known (see \cite{Mal89}) that  
$$
{\Exp}L^{\alpha}=\rho_{\alpha}(L^{p},L^{\infty}),
$$
where  $\rho_{\alpha}(L^{p},L^{\infty})$ is the Calderón--Lozanovski{\v \i} space defined by
\[
\rho_{\alpha}(s,t)=t\phi^{-1}\Big(\frac{s}{t}\Big),
\]
where
\[
\phi(u)=(e^{u^{\alpha}}-1)^{1/p}.
\]
By the Bukhvalov's theorem \cite[Theorem 6 b)]{Bu87} we get 
\[
\rho_{\alpha}(L^{\infty}(L^{p}),L^{\infty}(L^{\infty}))=L^{\infty}\left({\Exp}L^{\alpha}\right) {\rm \ and\ }\rho_{\alpha}(L^{p}(L^{p}),L^{p}(L^{\infty}))=L^{p}\left({\Exp}L^{\alpha}\right).
\]
In consequence, 
\[
\mathcal{T}:L^{\infty}\left({\Exp}L^{\alpha}\right)\to L^{p}\left({\Exp}L^{\alpha}\right),
\]
since the Calderón--Lozanovski{\v \i} construction is an interpolation method for positive operators (see \cite{Mal89}). 
\endproof

Having the above lemma Theorem \ref{main} follows. 

\proof[Proof of Theorem \ref{main}]
Fix $1<p<q<2$ and let $(f_i)_{i=1}^n\subset X_p$. We have by $q$-concavity of $L^p$
\[
\Big(\sum_i \|f_i\|_{L^p}^q\Big)^{1/q}\leq \Big\|\Big(\sum_i |f_i|^q\Big)^{1/q}\Big\|_{L^p}\leq M \Big\|\Big\|\sum_i r_i(s)f_i(t)\Big\|_{X_q(s)}\Big\|_{L^p(t)},
\]
where the second inequality and constant $M$ come from the fact that Rademacher system spans copy of $l^q$ in $X_q$ \cite[Example 3.1]{As20}. 

We claim that there is an Orlicz space ${\Exp}L^{\alpha}$ such that 
\[
X_p\subset {\Exp}L^{\alpha}\subset X_q.
\]
Suppose for the moment the claim is true. Then we have for some constants $d,D>0$
\[
\Big\|\Big\|\sum_i r_i(s)f_i(t)\Big\|_{X_q(s)}\Big\|_{L^p(t)}\leq D\Big\|\Big\|\sum_i r_i(s)f_i(t)\Big\|_{{\Exp}L^{\alpha}(s)}\Big\|_{L^p(t)}
\]
and using Lemma \ref{lemtransp2} 
\[
\leq 2D \Big\|\Big\|\sum_i r_i(t)f_i(s)\Big\|_{{\Exp}L^{\alpha}(s)}\Big\|_{L^{\infty}(t)}\leq 2dD \Big\|\Big\|\sum_i r_i(t)f_i(s)\Big\|_{X_p(s)}\Big\|_{L^{\infty}(t)}. 
\]
Finally, from the above inequalities it follows
\[
\Big(\sum_i \|f_i\|_{L^p}^q\Big)^{1/q}\leq 2dDM\max_{\epsilon_i=\pm 1}\Big\|\sum_i \epsilon_i f_i\Big\|_{X_p}.
\]
It remains to prove the claim. On the one hand, by \cite[Example 3.1]{As20}  we have that 
\[
X_p=(L^{\infty}, G)_{\theta,p}
\]
where $\theta=2(p-1)/p$. On the other hand, by \cite[Example 3.2 ]{As20}
\[
(L^{\infty}, G)_{\theta,\infty}={\Exp}L^{\alpha},
\]
where  $\alpha=2/\theta$. Then  for  $\eta=2(q-1)/q$
\[
X_p=(L^{\infty}, G)_{\theta,p}\subset (L^{\infty}, G)_{\theta,\infty}\subset (L^{\infty}, G)_{\eta,q}=X_q,
\]
since $\theta<\eta$. It means the claim holds with $\alpha=p/(p-1)$.
\endproof

\begin{Remark}
Let us firstly comment why the analogous arguing does not allow to capture the case $q=p$ in Theorem \ref{main}. 
Analyzing  the above  proof one easily realizes that considered inclusion would be $(p,1)$-absolutely summing if $\mathcal{T}$ was bounded from $L^{\infty}(X_p)$ to   $L^{p}(X_p)$. Then the naive thinking is to replace the Calder\'on--Lozanovski{\v \i} construction, which does not produce $X_p$  from the couple $(L^{\infty},L^p)$, by more general interpolation method, say, the real method, and repeat the argument. Unfortunately, it happens that unlike the Calder\'on--Lozanovski{\v \i} method,  the real method of interpolation does not commute with inner spaces of mixed norm construction. 
More precisely, we know that $X_p=(L^{\infty}, G)_{\theta,p}$
and that
$$
\mathcal{T}:L^{\infty}(G)\to   L^{p}(G)
$$
as well as
$$
\mathcal{T}:L^{\infty}(L^{\infty})\to   L^{p}(L^{\infty}),
$$
but we cannot conclude from the above that 
\[
\mathcal{T}:L^{\infty}(X_p)\to   L^{p}(X_p),
\]
because one cannot represent $L^{\infty}(X_p)$ as a space of the real method of interpolation for the couple under consideration, i.e. 
\[
L^{\infty}(X_p)=L^{\infty}((L^{\infty}, G)_{\theta,p})\not = (L^{\infty}(L^{\infty}), L^{\infty}(G))_{\theta,p}
\]
(see \cite[p. 288]{Cw74}).
\end{Remark}

It appears that not only limitations of the real method of interpolation  described in the remark above  do not allow to prove boundedness of  $\mathcal{T}:L^{\infty}(X_p)\to L^{p}(X_p)$. Actually, quite surprisingly,  $\mathcal{T}$ is not bounded, as indicates the following example, inspired by Bukhvalov \cite[p. 404]{BBS02}. 

\begin{Example}\label{trbr}
For each $1<p<2$
\[
\mathcal{T}:L^{\infty}(X_p)\not \to   L^{1}(X_p).
\]
\end{Example}
\proof

Let $n\in \mathbb{N}$ and let $a_m=e^{-(m^{1/\theta})+1}$ for $m=1,2,...,2^n$ and $\theta=p-1$. 
Then we put 
\[
\frac{1}{b_m}=\left\lceil e^{(m^{1/\theta})-1}\right\rceil=\left\lceil\frac{1}{a_m}\right\rceil.
\]
Thus $a_m\geq b_m$ and $2b_m\geq a_m$ since $\frac{1}{a_m}\geq 2$ for each $m=2,3,...,2^n$.
Let further 
$$
B_{m,k}^n=\left[\frac{m-1+(k-1)b_m}{2^n}, \frac{m-1+kb_m}{2^n}\right]
$$ 
and 
\[
A_{m,k}^n=[(k-1)b_m, kb_m],
\]
for $m=1,2,...,2^n$ and $k=1,2,...,\frac{1}{b_m}$. We will consider the following sequence of simple functions 
\[
K_n(s,t)=\sum_{m=1}^{2^n}W(a_m)^{-1/p}\sum_{k=1}^{1/b_m}\chi_{B^n_{m,k}}(t)\chi_{A_{m,k}^n}(s),
\]
where $W(t)=\ln^{1-p}(et^{-1})$. 
Then for each $t\in I$ there are unique $m$ and $k$ such that $t\in B_{m,k}^n$. This yields for each $t\in I$
\[
\|K_n(\cdot,t)\|_{X_p}=\left(W(a_m)^{-1}\int_0^{b_m}dW(t)\right)^{1/p}\leq 1.
\]
Thus $\|K_n\|_{L^{\infty}(X_p)}\leq 1$. 

 Consider now the image of $K_n$ with respect to $\mathcal{T}$, i.e. 
 \[
 \mathcal{T}K_n(s,t)=\sum_{m=1}^{2^n}W(a_m)^{-1/p}\sum_{k=1}^{1/b_m}\chi_{B^n_{m,k}}(s)\chi_{A_{m,k}^n}(t).
 \]
Now, it follows from the definitions of sets $B_{m,k}^n$, $A_{m,k}^n$ and functions $K_n$, that for each $n$, the distributions of functions $\mathcal{T}K_n(\cdot,t)$ are the same for each $t\in I$. In consequence, for arbitrary $t\in I$
 \[
 \| \mathcal{T}K_n\|_{L^{1}(X_p)}= \| \mathcal{T}K_n(\cdot,t)\|_{X_p}.
 \]

We have for sufficiently large $n$
\[
\|\mathcal{T}K_n(\cdot,t)\|_{X_p}^p=\int_0^{\infty}W\Big(|\{s\in I:K_n(s,t)^p>\lambda\}|\Big)d\lambda
\]
\[
\geq \sum_{m=2}^{2^n}\int_{W(a_{m-1})^{-1}}^{W(a_m)^{-1}} W\Big(|\{s\in I:K_n(s,t)^p>\lambda\}|\Big)d\lambda
\]
\[
\geq \sum_{m=2}^{2^n}\Big[W(a_m)^{-1}-W(a_{m-1})^{-1}\Big]W\Big(2^{-n}\sum_{j=m}^{2^n}b_j\Big)
\]
\[
=\sum_{m=2}^{2^n}W\Big(2^{-n}\sum_{j=m}^{2^n}b_j\Big)\geq \sum_{m=2}^{2^n}W(2^{-n-1}a_m)
\]
\[
\geq \sum_{m=2}^{2^n}\ln^{1-p}(2^{n+1}e^{m^{1/\theta}})=\sum_{m=2}^{2^n}\Big((n+1){\ln 2}+m^{1/\theta}\Big)^{-\theta}
\]
\[
\geq \sum_{m=2}^{2^n}\frac{1}{(n+m^{1/\theta})^{\theta}}\geq \sum_{m=2}^{2^n}\frac{1}{n^{\theta}+m}\approx n-\theta\ln n\to \infty.
\]
\endproof

Concluding finally the discussion on the proof of Theorem \ref{main} we point out that actually we do not need  boundedness of $\mathcal{T}$ on the whole space $L^{\infty}(X_p)$, but rather on its subspace 
\[
Y=\overline{{\rm span}\Big\{\sum_k r_k(t)f_k(s):f_k\in X_p\Big\}}^{L^{\infty}(X_p)},
\]
and the example above does not answer the question if $\mathcal{T}$ is bounded from $Y$ to $L^{p}(X_p)$. 
At the same time, from  \cite{Cw74} it follows that 
\[
L^{p}(X_p)=L^{p}((L^{\infty}, G)_{\theta,p})= (L^{p}(L^{\infty}), L^{p}(G))_{\theta,p}
\]
and
\[
 (L^{\infty}(L^{\infty}), L^{\infty}(G))_{\theta,p} \subset L^{\infty}((L^{\infty}, G)_{\theta,p})= L^{\infty}(X_p). 
\]
Thus one could also conclude that $\mathcal{T}:Y\to L^{p}(X_p)$, if there holds
\begin{equation}\label{inc1}
Y\subset  (L^{\infty}(L^{\infty}), L^{\infty}(G))_{\theta,p}.
\end{equation}

Regardless boundedness of $\mathcal{T}$, the question remains if Theorem \ref{main} holds true with $q=p$.

\begin{conjecture}\label{c1}
Let $1<p<2$. Then the inclusion $X_p\subset L^p$ is $(p,1)$-absolutely summing. 
\end{conjecture}

The next trivial example explains that the choice of $L^p$ in the conjecture is optimal, i.e. $L^p$ cannot be replaced by any smaller Lebesgue space.

\begin{Example}\label{example11}
For $1<p<q$ the inclusion $X_p\subset L^q$ is not $(p,1)$-summing if $q>p$.
\end{Example}
\proof
Let $f_i^k=\chi_{A_i^k}$, where $A_i^k=\left[\frac{i-1}{k},\frac{i}{k}\right)$.
Then 
\[
\max_{\epsilon_i=\pm 1}\Big\|\sum_{i=1}^k \epsilon_i f^k_i\Big\|_{X_p}=1.
\]
On the other hand
\[
\Big(\sum_{i=1}^k \|f^k_i\|_q^p\Big)^{1/p}=k^{\frac{1}{p}-\frac{1}{q}}.
\]
\endproof

%%%%%%%%%%%%%%%%%%%%%%%%%%%%%%%%%%%%%%%%%%%%%%%%%%%%%%%%%%%%%%%%%%%%%%%%%%%%%%%%%%%%%%%%%%%%%%%%%%

Having in mind that the transposition operator acts ``nicely'' on mixed norm spaces with inner Orlicz space, we consider summability of an analogous inclusion but with  some exponential Orlicz space instead of the Lorentz space $X_p$. This time the difficulty appears on the ``concavity'' side and the following theorem will be used as a replacement of $q$-concavity of $L^p$ space ($p\leq q$) from the proof of Theorem \ref{main}.

\begin{Theorem}\label{concave}
Let $1\leq p<q<\infty$. Then there is $C>0$ such that for arbitrary finite collection $(f_k)_{k=1}^n\subset L^p$ there holds 
\begin{equation}\label{concave1}
\left\|\left(\|f_k\|_{L^p}\right)_k\right\|_{l^{q,\infty}}\leq C \left\|\left\|(f_k)_k\right\|_{l^{q,\infty}}\right\|_{L^p}.
\end{equation}
\end{Theorem}
\begin{Remark}
Notice that the claim of Theorem \ref{concave} may be read as: the space $L^p$ is $l^{q,\infty}$ concave for each $1\leq p<q<\infty$. In fact, \eqref{concave1} is a direct generalization of the classical definition of concavity, where $l^q$ space is replaced by  $l^{q,\infty}$. While convexity and concavity of Banach lattices have been widely investigated and explored for decades, it seems that such a generalization,  hasn't been considered so far, up to our knowledge. 
\end{Remark}
\proof[Proof of Theorem \ref{concave}] Fix  $1\leq p<q<\infty$. We will prove formally more, namely that 
\[
L^p(l^{q,\infty})\subset l^{q,\infty}(L^p),
\]
which applied for finite sequences of functions from $L^p$ implies the claim.
Recall that the space $L^p$ is $r$-concave for each $r\geq p$ (see \cite[p. 46]{LT79}). It gives that for $r$ and $s$ such that $p<s<q<r$, there holds
\[
L^p(l^r)\subset l^r(L^p)\ \ \ and\ \ \ L^p(l^s)\subset l^s(L^p),
\]
(see \cite[p. 46]{LT79}). 

Applying \cite[Thm. 1, p. 125]{Tr78} with $\frac{1}{q}=\frac{1-\theta}{r}+\frac{\theta}{s}$ we have
\[
(l^r(L^p),l^s(L^p))_{\theta,\infty}=l^{q,\infty}(L^p).
\]
On the other hand, by \cite[p. 288]{Cw74}, since $p<\infty$ (the second index of the real method is bigger than $p$ from the outer space), 
\[
L^p((l^r,l^s)_{\theta,\infty})\subset (L^p(l^r),L^p(l^s))_{\theta,\infty}.
\]
Finally, 
\[
L^p(l^{q,\infty}) =L^p((l^r,l^s)_{\theta,\infty})\subset (L^p(l^r),L^p(l^s))_{\theta,\infty}\subset (l^r(L^p),l^s(L^p))_{\theta,\infty}=l^{q,\infty}(L^p),
\] 
as claimed.
\endproof

\proof[Proof of Theorem \ref{mainOrlicz}]
Fix $1<p<q<2$ and let $(f_k)_{k=1}^n\subset {\Exp}L^{q'}$. We have by Theorem \ref{concave} and \cite[Example 3.2]{As20}
\[
\Big\| (\|f_k\|_{L^p})_k\Big\|_{l^{q,\infty}}\leq C \Big\| \Big\|(f_k(t))_k\Big\|_{l^{q,\infty}}\Big\|_{L^p(t)} 
\leq  CM \Big\|\Big\|\sum_k r_k(s)f_k(t)\Big\|_{{\Exp}L^{q'}(s)}\Big\|_{L^p(t)}.
\]
Applying now Lemma \ref{lemtransp2}  we get 
\[
\Big\|\Big\|\sum_k r_k(s)f_k(t)\Big\|_{ {\Exp}L^{q'}(s)}\Big\|_{L^p(t)}\leq D\Big\|\Big\|\sum_k r_k(t)f_k(s)\Big\|_{ {\Exp}L^{q'}(s)}\Big\|_{L^{\infty}(t)}
\]
and finally
\[
\Big\| (\|f_k\|_{L^p})_k\Big\|_{l^{q,\infty}}\leq CMD\max_{\epsilon_k=\pm 1}\Big\|\sum_k \epsilon_k f_k\Big\|_{ {\Exp}L^{q'}}.
\]
\endproof

Actually, following interpolation intuitions and in the light of Example \ref{example11} one could extend Conjecture \ref{c1} even further.

\begin{conjecture}\label{c3}
Let $E$ be an exact interpolation space between $l^1$ and $l^2$ given by the formula $E=(l^1,l^2)_S^K$ and let $X_E=(L^{\infty},G)_S^K$. Then embedding $X_E\subset (L^1,L^2)_S^K$ is $(E,1)$-absolutely summing.
\end{conjecture}

Notice that this agrees with Conjecture \ref{c1} when $E=l^p$, $1<p<2$. 
%%%%%%%%%%%%%%%%%%%%%%%%%%%%%%%%%%%%%%%%%%%%%%%%%%%%%%%%%%%%%%%%%%%%%%%%%%%%%%%%%%%%%%%%%%%%%%%%%%%%%%%%%%%%%%%%%%%%%%%%%%%%%%%%%%%%%%%%%%%%%%%%%%%%%%%%%%%%%%%%

\section{Inclusions into $L^1$}\label{AS}

We begin with a sequence of technical results that will lead to the proof of Theorem \ref{Th1}. 

\begin{Lemma} \label{LL2}
For each constant $0<d<\frac{1}{2}$ and for any $N$ and $(e_k)_{k=1}^{N},(c_k)_{k=1}^{N}\subset \mathbb{R}_+$ equalities 
\[
\sum_{k=1}^{N}e_k=1\ \ and \ \ \sum_{k=1}^{N}c_ke_k=N
\]
imply that there is $1\leq k_0\leq N$ such that:
\begin{center}
($c_{k_0}e_{k_0}>d$ and $e_{k_0}< N^{-1}$) \  or ($c_{k_0}>dN$ and $e_{k_0}\geq  N^{-1}$).\end{center}
\end{Lemma}
\proof
Let $d< \frac{1}{2}$ and $(e_k)_{k=1}^{N},(c_k)_{k=1}^{N}\subset \mathbb{R}_+$ be arbitrary and suppose for the contrary that for each $1\leq k\leq N$:
\begin{center}
(if $e_{k}< N^{-1}$ then $c_{k}e_{k}\leq d$)  and (if $e_{k}\geq  N^{-1}$ then $c_{k}\leq dN$).
\end{center}
We have 
\[
N=\sum_{k=1}^{N}c_ke_k=\sum_{\{k:e_{k}< N^{-1}\}}c_ke_k +\sum_{\{k:e_{k}\geq N^{-1}\}}c_ke_k 
\]
\[
\leq dN + \sup_{\{k:e_{k}\geq N^{-1}\}}c_k\sum_{\{k:e_{k}\geq N^{-1}\}}e_k\leq dN+dN=2dN<N.
\]
The contradiction proves the claim. 
\endproof

\begin{Lemma}\label{L1}
For arbitrary $0<d<\frac{1}{2}$, $i\in\mathbb{N}$  and any $y_1,\dots,y_i\in L^1$ there is a selection of signs $\eta'_l=\pm 1$, $l=1,\dots,i$, such that
\[
 d2^{-i}\sum_{l=1}^i\|y_l\|_1\leq \int_0^{2^{-i}}\Big(\sum_{l=1}^i\eta'_ly_l\Big)^*(t)\,dt.
\]
\end{Lemma}
\proof%[Proof of Lemma \ref{L1}]
Without lost of generality we may assume that 
\[
\sum_{l=1}^i\int_{0}^1|y_l(t)|\,dt=2^i.
\]
Further, for any sign arrangement $\eta=(\eta_l)_{l=1}^i$ we define the sets
$$
E_\eta:=\{t\in [0,1]:\,{\rm sign}\,y_l(t)=\eta_l,\,l=1,\dots,i\}
$$
(for definiteness, we set ${\rm sign}\,0=1$). 
Then, $\cup_{\eta} E_\eta=[0,1]$ and $E_{\eta^1}\cap E_{\eta^2}=\emptyset$ if $\eta^1\ne\eta^2$.
In consequence,
\[
\sum_{\eta}|E_{\eta}|=1.
\]
and
\[
\sum_{\eta}\int_{E_{\eta}}\Big(\sum_{l=1}^i \eta_l y_l(t)\Big)\,dt\\
=\sum_{l=1}^i\int_{0}^1|y_l(t)|\,dt=2^i.
\]
Since there is $2^i$ collections of signs $\eta$ we can apply Lemma \ref{LL2} with $N=2^i$, $e_k=|E_{\eta}|$ and 
\[
c_k=\frac{1}{|E_{\eta}|}\int_{E_{\eta}}\Big(\sum_{l=1}^i \eta_l y_l(t)\Big)\,dt.
\]
In consequence there is $\eta'$ such that:\\
(a) $|E_{\eta'}|\leq 2^{-i}$ and $\int_{E_{\eta'}}\Big(\sum_{l=1}^i \eta'_l y_l(t)\Big)\,dt\geq d$\\
or\\
(b)$|E_{\eta'}|> 2^{-i}$ and $\frac{1}{|E_{\eta'}|}\int_{E_{\eta'}}\Big(\sum_{l=1}^i \eta'_l y_l(t)\Big)\,dt\geq d2^i$.

In the first case (a) we have
\[
\int_0^{2^{-i}}\Big(\sum_{l=1}^i\eta'_ly_l\Big)^*(t)\,dt\geq\int_{E_{\eta'}}\Big(\sum_{l=1}^i \eta'_l y_l(t)\Big)\,dt\geq d2^{-i}\sum_{l=1}^i\int_{0}^1|y_l(t)|\,dt,
\]
as claimed.

In the second case (b) there holds
\[
\int_0^{2^{-i}}\Big(\sum_{l=1}^i\eta'_ly_l\Big)^*(t)\,dt\geq \frac{2^{-i}}{|E_{\eta'}|}\int_{E_{\eta'}}\Big(\sum_{l=1}^i \eta'_l y_l(t)\Big)\,dt \geq d2^{-i}\sum_{l=1}^i\int_{0}^1|y_l(t)|\,dt,
\]
thus the lemma is proved. 
\endproof

\begin{proof}[Proof of Theorem \ref{Th1}]
Let  $(g_k)_{k=1}^n\in L_1$ and to simplify further notation define $$a_k:=\|g_k\|_1 {\rm \ and \ } f_k:=\frac{g_k}{\|g_k\|_1}.$$
First, since the Rademacher functions $r_k$, $k=1,2,\dots,n$, are independent and also symmetrically and identically distributed, we may  assume that $a_1\ge a_2\ge\dots\ge a_n>0$. Moreover, thanks to concavity of the function $F(\tau)=\int_0^\tau y^*(t)\,dt$, where $y\in L_1$, it suffices to consider only the case when $\tau=2^{-i}$, $i\in\mathbb{N}$.

According to \cite[Corollary~1]{As99} (see also \cite[Corollary~7.2]{As20}), with universal constants we have that
\begin{equation}\label{equ5}
\int_0^{2^{-i}}\Big(\sum_{k=1}^n a_kr_k\Big)^*(t)\,dt\asymp 2^{-i}\Big\{\sum_{k=1}^i a_k+\sqrt{i}\Big(\sum_{k=i+1}^n a_k^2\Big)^{1/2}\Big\},\;\;i\in\mathbb{N}.
\end{equation}

Suppose that $1\le i< n$. The main idea of the proof is a splitting of the set of coefficients $a_k$, $k=1,2,\dots,n$, into $i$ groups so that each of them contains one of the largest elements $a_1,\dots,a_i$.

Let $m_0:=i$ and
$$
\sigma:=\Big(\sum_{k=i+1}^n a_k^2\Big)^{1/2}.$$
If now $m_l$, $l=1,2,\dots,r$, is defined inductively by the formula
\begin{equation}\label{equ5a}
m_l:=\max\Big\{m\le n:\,\sum_{k=m_{l-1}+1}^{m_l} a_k^2\le \frac{\sigma^2}{{i}}\Big\},
\end{equation}
then clearly $r\ge i$. We form the sets
$$
A_l:=\{j\in\mathbb{N}:\,m_{l-1}<j\le m_l\}\cup\{l\},\;l=1,\dots,i-1,$$
and
$$A_i:=\{j\in\mathbb{N}:\,m_{i-1}<j\le n\}\cup\{i\}.$$
As it is easily to see, $\bigcup_{l=1}^i A_l=\{1,2,\dots,n\}.$
Since $\|f_k\|_1=1$, $k=1,\dots,n$, then applying successively Fubini theorem, Khintchine (see \cite{Sz76}) and Minkowski inequalities, we get
\begin{eqnarray*}
\int_0^1\int_0^1\Big|\sum_{k\in A_l}a_kr_k(s)f_k(t)\Big|\,dt\,ds &=&
\int_0^1\int_0^1\Big|\sum_{k\in A_l}a_kr_k(s)f_k(t)\Big|\,ds\,dt\\ &\ge&\frac{1}{\sqrt{2}}\int_0^1\Big(\sum_{k\in A_l}a_k^2f_k(t)^2\Big)^{1/2}\,dt\\ &\ge&\frac{1}{\sqrt{2}}\Big(\sum_{k\in A_l}a_k^2\Big(\int_0^1 |f_k(t)|\,dt\Big)^2\Big)^{1/2}\\
&=& \frac{1}{\sqrt{2}}\Big(\sum_{k\in A_l}a_k^2\Big)^{1/2},\;\;l=1,\dots,i.
\end{eqnarray*}
Therefore, there are $\delta_k=\pm 1$, $k=1,2,\dots,n$, satisfying the inequality
\begin{equation}\label{equ8}
\int_0^1\Big|\sum_{k\in A_l}\delta_ka_kf_k(t)\Big|\,dt\ge
\frac{1}{\sqrt{2}}\Big(\sum_{k\in A_l}a_k^2\Big)^{1/2},\;\;l=1,\dots,i.
\end{equation}

Let us define the functions
$$
y_l:=\sum_{k\in A_l}\delta_ka_kf_k,\;\;l=1,\dots,i,$$
and also, for any sign arrangement $\eta=(\eta_l)_{l=1}^i$, the sets
$$
E_\eta:=\{t\in [0,1]:\,{\rm sign}\,y_l(t)=\eta_l,\,l=1,\dots,i\}
$$
(for definiteness, we set ${\rm sign}\,0=1$). 
Then, applying Lemma \ref{LL2} for $\epsilon_j =\delta_k \eta'_l$, $k\in A_l$, $l=1,\dots,i$, and $d=1/3$, by \eqref{equ8}, we have
\begin{eqnarray}
\int_0^{2^{-i}}\Big(\sum_{j=1}^n\epsilon_ja_jf_j\Big)^*(t)\,dt &=& \int_0^{2^{-i}}\Big(\sum_{l=1}^i\eta'_l\sum_{k\in A_l}\delta_ka_kf_k\Big)^*(t)\,dt=
\int_0^{2^{-i}}\Big(\sum_{l=1}^i\eta'_ly_l\Big)^*(t)\,dt\nonumber\\
&\geq& \frac{2^{-i}}{3}\sum_{l=1}^i\int_{0}^1|y_l(t)|\,dt \ge
\frac{2^{-i}}{3\sqrt{2}}\sum_{l=1}^{i}\Big(\sum_{k\in A_l} a_k^2\Big)^{1/2}.
\label{equ9}
\end{eqnarray}

Next, depending on the coefficients $a_j$, $j=1,2,\dots,n$, we consider two cases. Assume first that 
\begin{equation}\label{equ6}
\sum_{j=1}^i a_j \ge \frac12 \sqrt{i}\Big(\sum_{k=i+1}^n a_k^2\Big)^{1/2}.
\end{equation}
Recall that, by definition, $l\in A_l$ for each $l=1,\dots,i$. Consequently, from \eqref{equ9} it follows
$$
\int_0^{2^{-i}}\Big(\sum_{k=1}^n \epsilon_k a_kf_k\Big)^*(t)\,dt \ge 
\frac{2^{-i}}{3\sqrt{2}}\sum_{j=1}^{i} a_j.$$

Thus, in view of \eqref{equ5} and \eqref{equ6}, in this case estimate \eqref{equ4} (with a suitable universal constant) follows.

Conversely, suppose we have 
\begin{equation}\label{equ10}
\sum_{k=1}^i a_k < \frac12 \sqrt{i}\Big(\sum_{k=i+1}^n a_k^2\Big)^{1/2}.
\end{equation}
Then, since $a_j\le \frac{1}{i}\sum_{k=1}^i a_k$ for all $i<j\le n$, it holds
\begin{equation*}\label{equ11}
a_j\le\frac{\sigma}{2\sqrt{i}},\;\;i<j\le n.
\end{equation*}
Therefore, by the definition \eqref{equ5a} of the numbers $m_l$, we have
$$
\sum_{k\in A_l} a_k^2\ge\sum_{k=m_{l-1}+1}^{m_l} a_k^2\ge\frac{\sigma^2}{{i}}-\frac{\sigma^2}{4{i}}=\frac{3\sigma^2}{4{i}},\;\;l=1,\dots,i,
$$
whence
$$
\sum_{l=1}^i\Big(\sum_{k\in A_l} a_k^2\Big)^{1/2}\ge\frac{\sqrt{3}}{2}\sigma \sqrt{i}.
$$
Combining this together with inequality \eqref{equ9} and the definition of $\sigma$, we obtain
\begin{eqnarray*}
\int_0^{2^{-i}}\Big(\sum_{k=1}^n\epsilon_ka_kf_k\Big)^*(t)\,dt \ge
\frac{2^{-i}\sqrt{3}}{6\sqrt{2}}\sigma \sqrt{i}=\frac{2^{-i}\sqrt{3}}{6\sqrt{2}}\sqrt{i}\Big(\sum_{k=i+1}^{n} a_k^2\Big)^{1/2}.
\end{eqnarray*}
This estimate, \eqref{equ5} and \eqref{equ10} imply \eqref{equ4} (with a suitable universal constant). 

Thus, if $1\le i< n$, the theorem is proved. Regarding the case $i\ge n$,  it suffices to note that, by the definition of the Rademacher functions, we have 
$$
\int_0^{2^{-i}}\Big(\sum_{k=1}^n a_kr_k\Big)^*(t)\,dt= 2^{-i}\sum_{k=1}^n a_k,$$
and the proof is completed.
\end{proof}

%%%%%%%%%%%%%%%%%%%%%%%%%%%
Before we apply Theorem \ref{Th1} let us make some comments. Recall that having two functions $f,g\in L^1$ one says that $f$ is {\it submajorized} (in a sense of Hardy--Littlewood--P\'olya) by $g$ if for each $0<t\leq 1$
\[
\int_0^tf^*(s)ds\leq \int_0^tg^*(s)ds.
\]
This is usually denoted by $f\prec g$. Such a partial order plays an important role in the interpolation theory and  in each r.i. Banach function space $X$ with the Fatou property relation $f\prec g$ and $g\in X$ implies that also $f\in X$ and $\|f\|_X\leq \|g\|_X$ (we refer to \cite{BS88, KPS82} for more details). From this point of view inequality \eqref{equ4} may be seen a weak version of submajorization. More precisely, instead of having collection of signs verifying 
\[
\gamma\sum_{k=1}^n \|g_k\|_1r_k\prec \sum_{k=1}^n \epsilon_k g_k,
\]
(which cannot hold in general), we proved much less,  because  the collection of signs on the right depends on $0<\tau<1$. Yet, it appears to be enough to conclude the corresponding norm inequality in most of r.i. spaces.

\begin{Theorem}\label{Th2}
 Let $\varphi$ be an increasing concave function on $[0,1]$, $\varphi(0)=0$ and assume that  $X$ is an r.i. Banach function space such that $G\subset X$. There is a constant $c>0$ such that for all $n\in\mathbb{N}$ and
\begin{itemize}
\item[(i)] for each $(f_k)_{k=1}^n\subset M_{\varphi}$
\begin{equation}\label{equ12}
\Big\|\sum_{k=1}^n \left\|f_k\right\|_{L^1}r_k\Big\|_{ M_{\varphi}}\leq c\max_{\epsilon_k=\pm 1}\Big\|\sum_{k=1}^n \epsilon_k f_k\Big\|_{ M_{\varphi}},
\end{equation}
\item[(ii)] for each $(f_k)_{k=1}^n\subset X$
\begin{equation}
\Big\|\sum_{k=1}^n \left\|f_k\right\|_{L^1}r_k\Big\|_{X}\leq c\max_{\epsilon_k=\pm 1}\Big\|\sum_{k=1}^n \epsilon_k f_k\Big\|_{X}.
\end{equation}
\end{itemize}
\end{Theorem}

\begin{proof}
(i) By the definition of the norm in $M_{\varphi}$, there exists $\tau_0\in (0,1]$ such that 
$$
\Big\|\sum_{k=1}^n \left\|f_k\right\|_{L^1}r_k\Big\|_{M_{\varphi}}\le \frac{2\varphi(\tau_0)}{\tau_0}\int_0^{\tau_0}\Big(\sum_{k=1}^n \left\|f_k\right\|_{L^1}r_k\Big)^*(t)\,dt.$$
Furthermore, according to Theorem \ref{Th1}, we can find $\epsilon_k=\pm 1$, $k=1,2,\dots,n$, so that
$$
\frac{2\varphi(\tau_0)}{\tau_0}\int_0^{\tau_0}\Big(\sum_{k=1}^n \left\|f_k\right\|_{L^1}r_k\Big)^*(t)\,dt
\le\frac{2\varphi(\tau_0)}{\gamma\tau_0}\int_0^{\tau_0}\Big(\sum_{k=1}^n \epsilon_k f_k\Big)^*(t)\,dt
\le\frac{2}{\gamma}\Big\|\sum_{k=1}^n \epsilon_k f_k\Big\|_{M_{\varphi}}.
$$
Combining the last inequalities, we get \eqref{equ12}.

(ii) Denote by $\varphi_X$ the fundamental function of $X$. Then, 
$X\subset M_{\varphi_X}$ \cite[Theorem II.5.7]{KPS82}. By hypothesis, the Rademacher sequence is equivalent to the unit vector basis of $l_2$ both in $X$ and in $M_{\varphi_X}$ \cite{RS75}, i.e.,
$$
\Big\|\sum_{k=1}^n a_kr_k\Big\|_{X}\asymp \Big\|\sum_{k=1}^n a_kr_k\Big\|_{M_{\varphi_X}}\asymp\|(a_k)\|_{l_2}.$$
Applying point (i) there are $\epsilon_k=\pm 1$, $k=1,2,\dots,n$, satisfying
$$
\Big\|\sum_{k=1}^n \left\|f_k\right\|_{L^1}r_k\Big\|_{M_{\varphi_X}}\leq c \Big\|\sum_{k=1}^n \epsilon_k f_k\Big\|_{M_{\varphi_X}}\leq c \Big\|\sum_{k=1}^n \epsilon_k f_k\Big\|_{X}.
$$
Summarizing the last relations, we obtain the desired inequality.
\end{proof}

%Corollary \ref{cor3}(i) implies sharp absolute summability type estimates for canonical inclusions of Marcinkiewicz spaces into $L_1$. Let ${\rm Exp}\,L^q$ be the Orlicz space, generated by an Orlicz function equivalent to $e^{u^q}$ for large $u>0$. Since ${\rm Exp}\,L^q$ is an interpolation Marcinkiewicz space between the spaces $L_\infty$ and ${\rm Exp}\,L^2$ if $q>2$, by \cite{RS75} (see also \cite[Example~3.2]{As20}), we get 

\begin{Corollary}\label{mainW}
Let $1<p<2$ and $\frac{1}{p}+\frac{1}{q}=1$. Then  inclusions ${\Exp}L^{q}\subset L^1$  and $X_{p}\subset L^1$  are $(l^{p,\infty},1)$-absolutely summing.  
\end{Corollary}
\proof
Of course, it is enough to prove only the first part of the claim, since $X_{p}\subset {\Exp}L^{q}$. We need to show that 
 there exists a constant $c>0 $ such that  for all $n\in\mathbb{N}$ and $(f_k)_{k=1}^n\subset {\Exp} L^q$ 
\begin{equation*}
\sup_{1\le k\le n} k^{-1/q} \sum_{i=1}^k (\|f_i\|_{L^1})^*(i)\leq c\max_{\epsilon_k=\pm 1}\Big\|\sum_{k=1}^n \epsilon_k f_k\Big\|_{{\Exp} L^q}.
\end{equation*}
By \cite[Example 3.2]{As20}
\begin{equation}\label{123}
\sup_{1\le k\le n} k^{-1/q} \sum_{i=1}^k (\|f_i\|_{L^1})^*(i)\asymp \Big\|\sum_{k=1}^n \left\|f_k\right\|_{L^1}r_k\Big\|_{{\Exp} L^q}.
\end{equation}
However, the Orlicz space ${\Exp}L^q$ coincides with the Marcinkiewicz space $M_{\varphi_p}$ for $\varphi_p(u)=\log^{1/p-1}(e/u)$, which means that we can apply Theorem \ref{Th2}(i) to get 
\begin{equation}\label{234}
\Big\|\sum_{k=1}^n \left\|f_k\right\|_{L^1}r_k\Big\|_{{\Exp}L^q}\leq c\max_{\epsilon_k=\pm 1}\Big\|\sum_{k=1}^n \epsilon_k f_k\Big\|_{{\Exp}L^q}. 
\end{equation}
Summarizing both inequalities \eqref{123} and \eqref{234}, we obtain the claim. 
\endproof

%%%%%%%%%%%%%%%%%%%%%%%%%%%%%%%%%%%%%%%%%%%%%%%%%%%%%%%

\section{Applications to embeddings of Sobolev spaces}

Recall that it was Pełczyński's question to describe absolute summablity of Sobolev embeddings and it was partially answered by the third author in \cite{Woj97}. The main idea therein was to factorize the respective embedding by Besov spaces and apply Bennett--Carl theorem. 
Having in hand continuous version of the Bennett--Carl theorem, i.e. Theorem \ref{main}, we can also conclude absolute summability of embedding of Sobolev space into Lebesgue space in the critical case. 
Here we come to a coincidence of independent interest. Namely, spaces $X_p$ appears naturally not only in the context of our considerations, but also spaces $X_p$ are optimal r.i. spaces for Sobolev embedding in the critical case. Precisely, for $m,n\in \mathbb{N}$, $2\leq m<n<2m$ there holds 
\begin{equation}\label{optimalsobolev}
W^{\frac{n}{m},m}(I^n)\subset X_{\frac{n}{m}}(I^n)
\end{equation}
and $ X_{\frac{n}{m}}(I^n)$ is the smallest r.i. space verifying inclusion \eqref{optimalsobolev}  (\cite{CP98},  cf. \cite[Example 4.2]{Pi09}).

\begin{Theorem}
Let $2\leq m<n<2m\in \mathbb{N}$. Then the inclusion 
\[
W^{\frac{n}{m},m}(I^n)\subset L^{\frac{n}{m}}(I^n)
\]
is $(q,1)$-absolutely summing for each $\frac{n}{m}<q\leq 2$. 
\end{Theorem}

\proof We have 
\[
W^{\frac{n}{m},m}(I^n)\subset X_{\frac{n}{m}}(I^n) \subset L^{\frac{n}{m}}(I^n)
\]
and since, by Theorem \ref{main}, the second inclusion is $(q,1)$-absolutely summing for each $\frac{n}{m}<q\leq 2$, the claim follows.
\endproof

Conjecture \ref{c1} suggests the following counterpart for Sobolev embedding. 
\begin{conjecture}\label{c5}
Let $2\leq m<n<2m\in \mathbb{N}$. The inclusion 
\[
W^{\frac{n}{m},m}(I^n)\subset L^{\frac{n}{m}}(I^n)
\]
is $(\frac{n}{m},1)$-absolutely summing. 
\end{conjecture}

\section*{Acknowledgments and Declarations}

The second named author wishes to thank Professor Lech Maligranda for indicating some of references.

The work of the first named author on Theorem \ref{Th1} was supported by the Russian Science Foundation (project no. 23-71-30001) at Lomonosov Moscow State University.

The research of the second named author was supported by the National Science Center (Narodowe Centrum Nauki), Poland (project no.~2017/26/D/ST1/00060).

The research of the third named author was supported by the National Science Center (Narodowe Centrum Nauki), Poland (project no.~2020/02/Y/ST1/00072).

\end{document}